\documentclass[letterpaper, 11 pt]{article}
\usepackage{fullpage,amsthm, amsmath, amssymb, amsfonts}

\newtheorem{thm}{Theorem}
\newtheorem{prop}{Proposition}
\newtheorem{lem}{Lemma}
\newtheorem{defn}{Definition}
\newtheorem{cor}{Corollary}
\newtheorem{prb}{Problem}
\newcommand{\newword}[1]{\textbf{\emph{#1}}}

\usepackage{lipsum}

\begin{document}

\title{A Combinatorial Problem from Group Theory}

\author{Eugene Curtin and Suho Oh}
\maketitle




\begin{abstract}
Keller proposed a combinatorial conjecture on construction of an $n$-by-infinite matrix, which comes from showing the existence of many orbits of different sizes in certain linear group actions \cite{Keller2014}. He proved it for the case $n=4$, and we show that conjecture is true in the general case. We also propose a combinatorial game version of the conjecture which even further generalizes the problem.
\end{abstract}







\section{The original problem}

Keller conjectured \cite{Keller2014} that given any $n\times \infty$ matrix of $n$ element sets $(S_{i,j})$, 
it is possible to construct an  $n\times \infty$ matrix $(x_{i,j})$ satisfying the following conditions: 
\begin{itemize}
\item For all $i$ and $j$, $x_{i,j}\in S_{i,j}$.  
\item The first $n-2$ elements in each row are distinct and never repeated later in the row. 
\item For all $t$, the $n$ sets  $\{x_{i,1},x_{i,2}, \ldots ,x_{i,t}\}$ ($1$ set from each row $i$) are distinct. 
\end{itemize}

This problem comes from showing the existence of many orbits of different sizes in certain linear group actions. Keller showed that it is possible when $n=4$, and conjectured that it was true in general.

The problem is very easy to solve in certain special cases. When all $S_{i,j}$ are all pairwise disjoint, any choice of $x_{i,j}$'s would lead to a valid construction. When all $S_{i,j} = \{1, \ldots, n\}$, 

\[ \begin{array}{ccccccccc}
1 & 2 & \cdots & n-2 & n-1 & n-1 & \cdots & n-1 & \cdots \\
2 & 3 & \cdots & n-1 & n & n & \cdots & n & \cdots \\
3 & 4 & \cdots & n & 1 & 1 & \cdots & 1 & \cdots \\
4 & 5 & \cdots & 1 & 2 & 2 & \cdots & 2 & \cdots \\
\vdots & \vdots & \vdots & \vdots & \vdots & \vdots & \vdots \\
n & 1 & \cdots & n-3 & n-2 & n-2 & \cdots & n-2 & \cdots \end{array} \]

gives an easy solution. So the problem is trivial when all $S_{i,j}$ are pairwise disjoint, or when all of them are the same.

In Section~\ref{the proof}, we will show that the conjecture is true for any value of $n$. In Section~\ref{the game}, we propose a combinatorial game which generalizes this original problem.

\section{The proof of the problem}
\label{the proof}
We will first show that it is enough to prove when one has finitely many number of columns. 

\begin{lem}[K{\H o}nig]
\cite{Konig}
Let G be a connected graph with infinitely many vertices such that each vertex is adjacent to only finitely many other vertices. Then G contains a path with no repeated vertices that starts at one vertex and continues from it through infinitely many vertices.
\end{lem}

\begin{cor}
\label{Konig}
If one can show that the conjecture is true for all $n$-by-$k$ matrices, it implies that the conjecture is true for $n$-by-$\infty$ matrices.
\end{cor}
\begin{proof}
Fix $n$ and $S_{i,j}$'s. Let $G$ be a graph where the vertices are given by all possible solutions to the $n$-by-$k$ cases, where $k$ goes from $1$ to $\infty$. We connect two vertices $a,b$ by an edge if $a$ is a solution to the $n$-by-$(k-1)$ case, $b$ is a solution to the $n$-by-$k$ case, and $a$ is exactly $b$ with last column deleted. Then this graph has infinitely many vertices, and only has finitely many number of components (since there are only finitely many possible solutions to the $n$-by-$1$ case). Now pick a component that has infinitely many vertices, and apply K{\H o}nig's infinity lemma in order to get a non-repeated path of infinite length. This path gives a solution to the $n$-by-$\infty$ case.
\end{proof}

Now we will introduce some notation to make the proof smoother. 

\begin{defn}
We say that a sequence $\{a_i\}$ is a \newword{representative} of a sequence of sets $\{S_i\}$ if for each $i$, $a_i \in S_i$. Given two sequences $\{a_1,a_2,\ldots\}$ and $\{b_1,b_2,\ldots\}$, we say that they are \newword{equivalent} if $\{a_1,\ldots,a_k\} = \{b_1,\ldots,b_k\}$ for all $k$. And we will say that they are \newword{compatible} if $\{a_1,\ldots,a_k\} \not =  \{b_1,\ldots,b_k\}$ for all $k$. A representative is \newword{regular} if $x_j \not \in \{x_1, \dots ,x_{j-1} \}$  when $S_j \setminus \{x_1, \dots ,x_{j-1} \}$ has at least 2 elements.
\end{defn}

For example, the sequences $\{1,2,3,1,1,\ldots\}$ and $\{1,2,3,2,2,\ldots\}$ are equivalent. The sequences $\{1,2,3,4,4,\ldots\}$ and $\{2,3,4,5,5,\ldots\}$ are compatible. One can think of a regular sequence as a sequence where an element that has already appeared before may be chosen only if there aren't at least two new elements to choose from. Regularity of a sequence implies the second condition of Keller's original problem in Section 1. The following is the reformulation of Keller's problem, except that we relax the condition on the cardinality of the $S_{i,j}$'s. 

\begin{defn}
We say that an $n$-by-$k$ matrix where the entries are given by sets $S_{i,j}$ of cardinality at least $n$ is an \newword{$(n,k)$-system}. A \newword{solution} of this system is an $n$-by-$k$ matrix with entries $x_{i,j}$ such that:
\begin{itemize}
\item Each row of the solution is a regular representative of the corresponding row in the system.
\item All rows are pairwise compatible.
\end{itemize} 
\end{defn}

We now prove our main result.

\begin{thm}
There exists a solution to any $(n,k)$-system. Moreover, consider the $(n,1)$-system obtained by just looking at the $S_{i,1}$'s. Any solution to this $(n,1)$-system can be extended to a solution of the $(n,k)$-system.
\end{thm}
\begin{proof}
Set $S$ to be the union of all the $S_{i,j}$'s. We will first show that it is possible to assume that $\bigcup_{j\ge 2} S_{i,j} = S$ for all rows. If not, extend the $(n,k)$-system to an $(n,k+1)$-system where $S_{i,k+1}=S$ for all $i$. If we can find a solution to this $(n,k+1)$-system, we can just delete the last column to obtain a solution to the original $(n,k)$-system.

The statement is obviously true when $n=1$. For sake of induction, assume the statement is true for $n$, and we will show that it is also true for $n+1$. Pick an arbitrary solution $p$ of the $(n+1,1)$-system. For each $i \not = j$, we set $q_{i,j}$ as the least $t \geq 2$ such that $p_{i,1} \in S_{j,t}$. These numbers are well defined thanks to the assumption we made in the beginning of the proof. Properly reorder the rows so that $q_{1,2}$ is the minimum among the $q_{i,j}$'s. Set $a$ to be $p_{1,1}$. Set $d_1=1$, and for each $i \geq 2$, denote $d_i$ as the least $l \geq 2$ such that $a \in S_{i,l}$. (So $d_i=q_{1,i}$ for $i \geq 2$.)  Now reorder the rows $2,\ldots,n+1$ such that $d_2 \leq \cdots \leq d_{n+1}$. The positions $(i,d_i)'s$ will form a diagonal of sorts, encoding the earliest possible position where $a$ can appear in each row (except the first column). Let $b$ be $p_{n+1,1}$, the entry of the last row of the solution. 

Now in the $(n+1,k)$-system we are looking at, do the following:
\begin{itemize}
\item delete the last row,
\item for each row $i$, for all columns $t$ past the diagonal $d_i$ ($t \geq d_i$), remove $a$ from $S_{i,t}$,
\item for each row $i$, for all early columns $t$ before the diagonal ($t \leq min(d_i,n-1)$), remove $b$ from $S_{i,t}$.
\item and erase the entries at $(i,d_i)$.
\end{itemize}

(The description of the third operation might seem odd. Avoiding early occurrences of $b$ in the first $n$ rows of $x$ we get later on helps achieve compatibility with the eventual $(n+1)$-th row of $x$, and we allow later occurrences to maintain regularity.) Now we have an $(n,k-1)$-system. Set $y$ to be $y_{i,1} = p_{i,1}$ for $2 \leq i \leq n$, and set $y_{1,1}$ to be any element of $S_{1,2} \setminus \{a,y_{2,1}\ldots,y_{n,1}\}$. Then $y$ is a solution to the $(n,1)$-system, and we can extend this to a solution of the $(n,k-1)$-system via the induction hypothesis. Name this solution again as $y$. Insert $a$ into the $d_i$-th position of each row of $y$ to get $x$. We claim that $x$ is a solution to the $(n,k)$ system obtained from the original $(n+1,k)$-system by deleting the last row. The rows of $x$ are pairwise compatible since the rows of $y$ are pairwise compatible and $a$ appears exactly once per each row of $x$. Now assume for the sake of contradiction that some row $i$ of $x$ is not regular. This means there is some entry $j$ such that $x_{i,j} \in \{x_{i,1},\ldots,x_{i,j-1}\}$ and $|S_{i,j} \setminus \{x_{i,1},\ldots,x_{i,j-1}\}| \geq 2$. Now $j$ obviously can't be $ \geq d_i$ since $y$ is regular and $a$ appears only once per row. It also can't be between $1$ and $n-1$ since repeated elements in each row can only start appearing from column $n$. And for $j$ between $n$ and $d_i$, $S_{i,j}$ hasn't been changed so regularity of $y$ gives a contradiction. This finishes the proof of the claim that $x$ is a solution to the $(n,k)$ system obtained from the original $(n+1,k)$-system by deleting the last row.

All that remains now is to fill out the $(n+1)$-th row of $x$. If $d_{n+1} \leq n$, then $d_j\le n$ for all $j$ and $a$ must appear before $b$ in each of the first $n$ rows of $x$. In this case, any regular representative not using $a$ will suffice. When $d_{n+1} > n$, consider any regular representative of the last row up to the $n$-th position, $b=z_1,\ldots,z_n$. Each row of $x$ puts a partial ordering on $\{z_2,\ldots,z_n\}$ : look at the first position where $z_i$ appears in that row and order the $z_i$'s according to that position so that the ones appearing earlier are smaller. Extend that ordering into a total ordering by making the elements that didn't appear bigger than the ones that appeared. So we get $n-2$ total orderings from row $3$ to row $n$ in $x$. There is some symbol, say $z_t$, that is not the last in any of the $n-2$ orderings. Replace $z_t$ with any $z' \in S_{n+1,t} \setminus \{z_2,\ldots,z_n\}$, and extend the sequence to a regular representative of the last row of the system by imposing $x_{n+1,j} \not = z_t$ for $n+1 \leq j < d_{n+1}$ and $x_{n+1,j} \not = a$ for $j \geq d_{n+1}$. (This is possible since we are always allowed to avoid one particular element while constructing a regular representative.) 

Now we want to show that this sequence, the new $n+1$-th row of $x$, is compatible with rows $1$ to $n$ of $x$. For row $1$ it is obvious since row $1$ starts with $a$ and $(n+1)$-th row does not contain $a$ at all. Denote $\{x_{i,1},\ldots,x_{i,j}\}$ by $C(i,j)$. Assume for the sake of contradiction that $C(i,j) = C(n+1,j)$ for some $3 \leq i \leq n$. Since row $n+1$ starts with $b$ and contains no $a$, we need $b \in C(i,j)$ and $a\notin C(i,j)$, hence $j \geq n$. But this implies $\{z_2,\ldots,z_n\} \setminus \{z_t\} \subset C(n+1,j)=C(i,j)$. From the definition of $z_t$, it can't be the last element to occur among $\{z_2,\ldots,z_n\}$ in row $3$ to row $n$, so it has to be contained in $C(i,j)$. We get $z_t \in C(i,j) = C(n+1,j)$, which implies $j \geq d_{n+1}$. From $d_i \leq d_{n+1}$, we have $a \in C(i,j) = C(n+1,j)$, but we get a contradiction since $a$ never appears in the last row. The only case remaining to check is the compatibility against row $2$. Again assume for the sake of contradiction that $C(2,j) = C(n+1,j)$ for some $j$. Since $a$ never appears in the $(n+1)$-th row, $j < d_2$. But $b$ being the first element in the $n+1$-th row means $b \in C(2,d_2-1)$. This contradicts the choice of $a$, which should be an entry of $p$ that can be repeated the earliest among different rows. Therefore this finishes the proof of the claim that the $n+1$-th row of $x$ is compatible with all other rows of $x$. Hence $x$ is a solution of the $(n+1,k)$-system, and this proves the theorem by induction.



\end{proof}

As we have mentioned in Corollary~\ref{Konig}, combining this result with K{\H o}nig's Lemma implies the following result:
\begin{thm}
Keller's conjecture is true for any value of $n$.
\end{thm}

\section{Combinatorial game version}
\label{the game}
In this section we propose a combinatorial game. If a solution to such game always exists, it would imply the solution to Keller's problem. Recall that $[n]$ denotes $\{1,\cdots,n\}$.

\begin{prb}
Fix a natural number $n$. You start with $n$ sets $S_i = \{i\}$. Each round, for each $i \in [n]$, an adversary offers two elements from $[n] \setminus S_i$. Your goal is to pick one of the two elements offered, to add to $S_i$, such that all $S_i$'s are still pairwise different. Can you always make the correct choice for $n-2$ rounds? (So that $S_i$'s become the $n$ subsets of cardinality $n-1$ of the set $[n]$).
\end{prb}

It is easy to see that solving this problem would imply a solution to Keller's problem. A slightly weaker version of the problem can be described in terms of finding $n$ vertex disjoint paths in $n$ subposets of a Boolean lattice $B_n$. We say that a subposet $P$ of $B_n$ is \newword{branching} if every non-maximal element $A \in P$ is covered by at least $2$ elements of $P$ of cardinality $|A|+1$. We call the maximal elements of a branching poset to be the \newword{leaves} of the poset. A subposet of $B_n$ is \newword{rooted} if it contains the minimal element $\emptyset$. A subposet of $B_n$ is \newword{wide} if it contains all the singletons. A \newword{full path} is a simple path from some $\{i\}$ to some $[n] \setminus \{j\}$ in $B_n$.

\begin{prb}
\label{prb:posetprb}
Consider the Boolean lattice $B_n$. Let $P_1,\ldots,P_n$ be rooted branching subposets of $B_n$, such that $\{i\} \in P_i$ for all $i \in [n]$. Can one find a simple path for each $P_i$ so that they are all vertex disjoint?
\end{prb}

It is not hard to see that a solution to this problem would give a solution to Keller's problem: think of a sequence $a_1,\ldots,a_k$ as the set $\{a_1,\ldots,a_k\}$. So considering all possible sets you get from all possible sequences of a given row in Keller's problem gives us a subposet of $B_n$, where $n = |\bigcup_{i,j} S_{i,j}|$. This poset is a branching poset thanks to the regularity condition. 

We will show that this problem is easily solved when $P_1 = \cdots = P_n$. We denote this poset as $P$, and $P$ has to be wide. Assume for the sake of contradiction that we can't find $n$ vertex disjoint full paths in $P$. Using the following Menger's theorem tells us that there is a set $C$ of vertices in $P$ that separates the singletons from the complement of the singletons, with $|C| < n$.

\begin{thm}[\cite{Menger}]
Let $G$ be a graph with $A$ and $B$ two disjoint vertex subsets of $G$, with $|A| = |B| = n$. Then either $G$ contains $n$ pairwise vertex disjoint paths from $A$ to $B$, or there exists a set of fewer than $n$ vertices that separates $A$ from $B$.
\end{thm}

Take $C$ to be an inclusion-wise minimal set that separates the singletons from the complement of the singletons. Now consider all paths from $\emptyset$ to a vertex of $C$, that does not pass through other vertices of $C$. Then the union of these paths is again a wide rooted branching poset, and $C$ is exactly the set of maximal elements of this poset. Now we will show the following claim:

\begin{prop}
\label{propp}
Any wide rooted branching poset $P$ of $B_n$ has at least $n$ leaves.
\end{prop}

To do so we will be using the following lemma:

\begin{lem}
\label{lemm}
Any rooted branching subposet $P$ of $B_n$ has at least $h$ leaves, where $h$ is the maximal height of the subposet.
\end{lem}
\begin{proof}[Proof of Lemma~\ref{lemm}]
Let $P^x$ be the induced subposet below the singleton $x$, so that $P^x = \{y|y \geq_P \{x\}\}$. We use induction on the number of vertices to prove the lemma. When there is only one vertex, the statement is trivial. Assume for the sake of induction that the statement is true for subposets of size $< k$. Let $v_1,\ldots,v_r$ be the vertices at height $2$. Relabel them as $\{\{1\},\ldots,\{r\}\}$, such that $P^1$ has maximal height. Note that $P^1$ is isomorphic to a rooted branching subposet of $B_{n-1}$ of height $h-1$, so we may apply the induction hypothesis to obtain that $P^1$ has at least $h-1$ leaves. Now consider $P^2$. Starting from $\{2\}$, go down $P^2$ by avoiding the element $1$. Note that if $A \in P$ and $1 \not \in A$, then at most one set containing $1$ covers $A$, so branching guarantees we can avoid $1$. Thus we reach a leaf of $P^2$ which does not contain $1$ and thus cannot be in $P^1$. This proves that $P$ has at least $h$ leaves.

\end{proof}

Now we will prove the Proposition:
\begin{proof}[Proof of Proposition~\ref{propp}]
Assume for the sake of contradiction that we have at most $n-1$ leaves. We will show a contradiction by a fractional counting argument on the leaves. For each singleton $s$ and leaf $l$, define $f(s,l) = \frac{1}{|l|}$ if $s \leq_P l$, $0$ otherwise. Then $\sum_{s,l} f(s,l) \leq \sum_l |l| * \frac{1}{|l|} \leq n-1$. Since there are $n$ singletons, there has to exist some singleton $q$ such that $\sum_l f(q,l) <1$. Now consider $P^q$, the subposet below $q$. By Lemma~\ref{lemm}, $P^q$ has at least $h$ leaves where $h$ is the height of $P^q$. So $\sum_l f(q,l) = \sum_{l \in P^q} \frac{1}{|l|} \geq h * \frac{1}{h} = 1$, which leads to a contradiction.
\end{proof}




The above proposition tells us that $|C| < n$ cannot happen, and therefore has to be $n$ vertex disjoint full paths in $P$.

\begin{cor}
It is possible to find a set of vertex disjoint paths for Problem~\ref{prb:posetprb} in the case where $P_1 = \cdots =  P_n$.
\end{cor}






\bibliographystyle{plain}    
\bibliography{sample}

\end{document}